\documentclass[11pt]{article}

\usepackage[alphabetic]{amsrefs}
\usepackage{amsmath,amssymb,amsthm,enumerate}
\usepackage[all,cmtip]{xy}
\usepackage[applemac]{inputenc}

\def \1{\mathds{1}}
\def \A{{\mathbb A}}
\def \al{\alpha}
\def \bs{\backslash}

\def \CG{{\cal G}}

\def \CO{{\cal O}}

\def \eps{\varepsilon}
\def \F{{\mathbb F}}
\def \fin{\mathrm{fin}}
\def \Ga{\Gamma}
\def \GL{\operatorname{GL}}
\def \ga{\gamma}

\def \Id{{\rm Id}}

\def \inert{\mathrm{in}}
\def \La{\Lambda}
\def \la{\lambda}
\def \M{\operatorname M}

\def \N{{\mathbb N}}
\def \ol{\overline}

\def \PGL{\operatorname{PGL}}

\def \Q{{\mathbb Q}}
\def \R{{\mathbb R}}

\def \sm{\smallsetminus}

\def \tors{\mathrm{tors}}

\def \Z{{\mathbb Z}}
\def \({\left(}
\def \){\right)}

\newcommand{\e}
[1]{\emph{#1}\index{#1}}

\newcommand{\mat}
[4]{\(\begin{matrix}#1 & #2 \\ #3 & #4\end{matrix}\)}

\newtheorem{theorem}{Theorem}[section]

\newtheorem{lemma}[theorem]{Lemma}
\newtheorem{corollary}[theorem]{Corollary}

\begin{document}

\pagestyle{myheadings} \markright{Ihara zeta functions and class numbers}

\title{Ihara zeta functions and class numbers\\ \ \small\\
 Adv. Studies in Contemp. Math. Vol. 24, Issue 4, 439-450 (2014)}
\author{Anton Deitmar\thanks{This research was funded by the DFG grant DE 436/10-1}}
\date{}
\maketitle

{\bf Abstract:}
The prime geodesic theorem for cycles in Bruhat-Tits buildings is applied to unit groups of division algebras to derive new asymptotic assertions on class numbers of orders in imaginary quadratic fields.

$ $


$ $


$$ $$

\tableofcontents

\newpage
\section{Introduction}
The theory of the Selberg zeta function leads to asymptotic assertions about lengths of closed geodesics, known as \e{Prime Geodesic Theorems}, the first of which, for compact hyperbolic surfaces, is usually credited to Atle Selberg. In his 1970 Ph.D. thesis, Grigory Margulis generalised this result to surfaces of variable negative curvature, see \cites{Hejhal,KatokHassel}.
In number theoretical settings, more precise assertions, for instance estimates on remainder terms, are a vivid area of research, see \cites{Avdis,Sound}.
Also, higher rank generalizations are possible \cites{PGTHigher,PGTSL4}.
One possible application to number theory is in class number asymptotics \cites{Sarnak,class,classNC}.

In the nineteensixties, Yasutaka Ihara started investigating a $p$-adic analogue of the Selberg zeta function, by now known as the \e{Ihara zeta function} \cite{Ihara1,Ihara2}.
It can be interpreted as a geometric zeta function for the corresponding finite graph, which is a quotient of the Bruhat-Tits buildung attached to the $p$-adic group \cite{Serre}.
Over time, it has been generalized in stages by Sunada, Hashimoto and Bass \cites{Sun1,Sun2,Hash0,Hash1,Hash2,Hash3,Hash4,Bass,Sunada}.
Comparisons with number theory can be found in the papers of Stark and Terras \cites{ST1,ST2,ST3}.
In the $p$-adic case, there also is a prime geodesic theorem, see \cite{Terras}.
The present paper contains an application to class number asymptotics in the spirit of the above cited papers, which can be stated as follows:

Let $p$ be a prime number.
For any subring $\CO$ in a number field $F$ let $l_p(\CO)$ denote the infimum of the set
$$
\{ k\in\N: p^k\in N_{F/\Q}(\CO\sm\Q)\},
$$
where $N_{F/\Q}$ is the norm map and the infimum is $+\infty$ if no such $k$ exists.

Let $l$ be a prime number different from $p$ and such that $l\equiv 1 \mod(12)$ and let $I(l)$ be the set of all isomorphy classes of imaginary quadratic fields $F$ such that the prime $l$ is non-decomposed in $F$, i.e., there lies only one prime of $F$ above $l$.
Let $O_p(l)$ be the set of isomorphy classes of orders $\CO$ in some $F\in I(l)$ such that $\CO$ is maximal at  $l$ and at $p$.
Let $f_l(\CO)=f_l(F)\in\{ 1,2\}$ denote the inertia degree of $l$ in $F$.
Let $h(\CO[1/p])$ denote the class number of $\CO[1/p]$.
Let $\Delta$ be the greatest common divisor of the numbers $l_p(\CO)$ as $\CO$ ranges over $O_l(p)$.
Then every number $l_p(\CO)$ is of the form $\Delta m$ for some $m\in\N\cup\{+\infty\}$.

\begin{theorem}\label{thm1.1}
As $m\to \infty$, 
 one has the asymptotic assertion
$$
\sum_{\substack{\CO\in O_p(l)\\ l_p(\CO)=\Delta m}}f_l(\CO)h(\CO[1/p])\quad\sim\quad \frac{p^{\Delta m}}{2m}.
$$
\end{theorem}

Since the class number $h(\CO)$ is $\ge h(\CO[1/p])$ and $f_l(\CO)$ takes values  in $\{1,2\}$ one gets the following corollary.

\begin{corollary}
$$
\liminf_{m\to\infty}\frac m{p^{\Delta m}}\sum_{\substack{\CO\in O_p(l)\\ l_p(\CO)=\Delta m}}h(\CO)\ge \frac 14.
$$
\end{corollary}

It is the objective of the present paper, to give a proof of (a slightly more general version of) Theorem \ref{thm1.1}.

\section{Division algebras}
Let $R$ be an integral domain and $K$ its field of fractions.
Let $A$ be a finite-dimensional $K$-algebra with unit.
An \e{$R$-order} in $A$ is an $R$-subalgebra $\Lambda$ of $A$, which is finitely generated as $R$-module and spans the $K$-vector space $A$, i.e., $K\La=A$.

We are going to apply this in the following situations:
\begin{itemize}
\item $R=\Z$, $K=\Q$ in which case we simply say, $\La$ is an \e{order} in $A$,
\item $R=\Z[1/p]$, where $p$ is a prime number and $K=\Q$,
\item $R=\Z_p$ and $K=\Q_p$ for a prime $p$.
\end{itemize}

If the algebra $A$ is a number field $F$ and $\CO$ is an order in $F$, let $I(\CO)$ denote the set of all finitely generated $\CO$-submodules of $F$.
By the Jordan-Zassenhaus Theorem \cite{Reiner}, the set $[I(\CO)]$ of isomorphism classes of elements of $I(\CO)$ is finite.
Let $h(\CO)$ be its cardinality, called the \e{class number} of the order $\CO$.
Likewise, for a $\Z[1/p]$-order $\La$, the set $I(\La)$ of all isomorphism classes of finitely generated $\La$-modules in $F$ is finite and its cardinality $h(\La)$ is called the class number of $\La$.

If $F$ is a number field, there is a maximal order $\CO_F$, which is the integral closure of $\Z$ in $F$, and wich contains any other order in $F$.
The same applies in the local situation, if $F$ is a finite extension of $\Q_p$, the integral closure $\CO_F$ of $\Z_p$ is a maximal $\Z_p$-order containing any other order.
If $F$ is a number field and $p$ is a prime, then 
$F_p=F\otimes_\Q\Q_p$ is the product of $p$-adic fields.
We say that an order $\CO$ of $F$ is \e{maximal at $p$}, if $\CO\otimes_{\Z}\Z_p$ is the maximal order of $F_p$.

\begin{lemma}
Let $F$ be a number field.
\begin{enumerate}[\rm (a)]
\item An order $\CO$ in $F$ is maximal at $p$ if and only if $\CO=\CO_F\cap\CO[1/p]$.
\item 
The map $\psi:\CO\to\CO[1/p]$ is a bijection from the set of all orders $\CO$ in $F$, which are maximal at $p$, to the set of all $\Z[1/p]$-orders in $F$.
\item
If $F$ is quadratic and if two orders $\CO$ and $\CO'$ are isomorphic as rings, then they are equal. Further if both orders are maximal at $p$ and $\CO[1/p]\cong\CO'[1/p]$ then $\CO=\CO'$ as well.
\end{enumerate}
\end{lemma}

\begin{proof}
(a) Let $\CO$ be maximal at $p$ and let $y\in\CO_F\cap \CO[1/p]$.
Now $F_p=\prod_{j=1}^n F_{j,p}$, where each $F_{j,p}$ is complete in a valuation $v_j$ above $v_p$.
Now $y\in\CO_{F,p}=\CO_p=\prod_{j=1}^n\CO_j$, where $\CO_j$ is the valuation ring in $F_j$.
Then there exists a sequence $x_\nu$ in $\CO$ such that $x\nu$ converges to $y$ in the product topology of $\prod_{j=1}^n\CO_j$.
As all the valuations $v_j$ lie above $p$, we can assume, by switching to a subsequence if necessary, that $x_\nu-y$ is in $\CO_F$ divisible by $p^\nu$.
This means that $y\in \CO+p^\nu\CO_F$ for every $\nu\in\N$ and hence $y\in\CO$.
For the converse direction assume $\CO=\CO_F\cap\CO[1/p]$ and let $x\in\CO_{F,p}$. 
Then there exists $n\in\N$ such that $p^nx\in\CO_p$, which means $x\in\CO_p[1/p]$ and therefore
$x\in \CO_{F,p}\cap\CO_p[1/p]=(\CO_F\cap\CO[1/p])_p=\CO_p$.

(b) Part (a) implies injectivity of the map $\psi$.
For surjectivity, let $\Lambda$ be a $\Z[1/p]$-order in $F$ and let $\CO_F$ denote the maximal order in $F$, then $\CO=\CO_F\cap\Lambda$ is an order in $F$ which is maximal at $p$ by (a) and which has  the property that $\CO[1/p]=\La$.

(c)
A given  isomorphism $\phi:\CO\to\CO'$ of orders in $F$ lifts to a Galois-isomorphism $F\to F$, but any order of a quadratic field $F$ is of the form $\Z[n\al]$ for some $n\in\N$, where $\al=(d+\sqrt d)/2$ and $d$ is the discriminant of $F$.
Form this description one sees that any order is stable under the only non-trivial Galois homomorphism $\sqrt d\mapsto -\sqrt d$.
Hence $\CO=\CO'$.
Also, if $\CO[1/p]\cong\CO'[1/p]$ then, as $\CO$ is the integral closure of $\Z$ in $\CO[1/p]$, it follows that $\CO\cong\CO'$.
\end{proof}

Let $D$ denote a  division algebra over $\Q$ \cite{Pierce} of dimension $d^2$, where $d$ is a prime number, and let $D(\Z)$ denote a fixed maximal order in $D$.
Note that all maximal orders in $D$ are conjugate \cite{Reiner}.
For any ring $R$ we define
$$
D(R)=D(\Z)\otimes_\Z R.
$$
Then $D(\Q)$ is canonically isomorphic to $D$.
We write $D(R)^\times$ for the unit group and $D(R)^1$ for the group of all units of reduced norm one. The reduced norm induces an isomorphism
$$
D(R)^\times/D(R)^1\cong R^\times.
$$

For almost all prime numbers $p$, one has $D(\Q_p)\cong \M_d(\Q_p)$, where $\Q_p$ is the field of $p$-adic numbers and $\M_d(K)$ denotes the algebra of $d\times d$-matrices over a field $K$, \cite{Pierce}.
We include the case $p=\infty$ in which we write $\Q_p=\Q_\infty=\R$.
If $D(\Q_p)\cong \M_d(\Q_p)$, we say that $D$ \e{splits at $p$}.
If $D$ doesn't split at $p$, then $D(\Q_p)$ is a division algebra over $\Q_p$.

Let $S$ be the finite set of all primes, at which $D$ doesn't split.

\begin{lemma}\label{lem1.1}
Let $A\subset D$ be a $\Q$-subalgebra.
Then the dimension of $A$ is 1,$d$ or $d^2$.
In the first case $A=\Q$, in the last $A=D$.
In the remaining case $A$ is a field extension of $\Q$ of degree $d$, such that every $p\in S$ is non-decomposed in $A$, i.e., there is only one prime of $A$ above $p$.
Further, if $D(\R)$ is a division algebra, which can only happen for $d=2$, then the field $A$ is imaginary quadratic.
Every number field of degree $d$ satisfying these conditions occurs as a subalgebra of $D$.

If $d=2$ and $S$ contains a prime $l\equiv 1\mod(3)$ and a prime $m\equiv 1\mod (4)$, then the only torsion elements of the group $D(\Q)^\times$ are the  elements $\pm 1$.
\end{lemma}

If the set $S$ satisfies the condition of the second part of the lemma, we say that the primes $2$ and $3$ \e{decompose in} $D$.

\begin{proof}
The first part of the lemma is standard. It can be pieced together from the information in Pierce's book \cite{Pierce}.
A proof for $d\ge 3$, which also works for $d=2$ can be found in \cite{class}, Lemma 2.1.

For the second part let $d=2$ and let $a\in D(\Q)^\times_\tors$.
If $a\in\Q$, then $a=\pm 1$ and we are done.
Otherwise, the centralizer $F$ of $a$ in $D(\Q)$ is a quadratic number field.
Now note that the only quadratic fields that admit non-tivial roots of unity are $\Q(i)$ and $\Q(\sqrt{-3})$, where the torsion elements are the 4th and 6th roots of unity respectively.
These two fields are the only quadratic cyclytomic fields.
We have to make sure that the conditions on $S$ imply that these two fields do not embed into $D$.
Any prime $l\equiv 1\mod (3)$ is decomposed in the cyclotomic field $\Q(\sqrt{-3})$ by Proposition 8.5 in \cite{Neukirch}.
Any prime $m\equiv 1\mod (4)$ is decomposed in the Gaussian number field $\Q(i)$ by Theorem 1.4 in \cite{Neukirch}.
From this the claim follows.
\end{proof}

Let $p$ be a prime not in $S$.
Let $F/\Q$ be a number field of degree $d$ which embeds into $D(\Q)$.
Then for any embedding $\sigma:F\hookrightarrow D(\Q)$ the set
$$
\La_\sigma=\sigma^{-1}\(D(\Z[1/p])\)
$$
is a $\Z[1/p]$-order in $F$.

\begin{lemma}\label{lem1.2}
Let $p$ be a prime not in $S$.
Let $\sigma:F\hookrightarrow D(\Q)$ be an embedding of the degree $d$ number field $F$.
Then for any $l\in S$, the order $\La_{\sigma,l}=\La_\sigma\otimes_{\Z[1/p]}\Z_l$ is maximal in the local field $F_{l}$.
Conversely, let $\La\subset F$ be a $\Z[1/p]$-order such that for any $l\in S$ the order $\La_{l}$ is maximal, then there exists an embedding $\sigma$ such that $\La=\La_\sigma$.
\end{lemma}

\begin{proof}
Let $\CO_\sigma=\sigma^{-1}\(D(\Z)\)$, then $\CO_\sigma$ is an order in $F$ with $\CO[1/p]=\La_\sigma$.
Note that $\La_{\sigma,l}=\CO_\sigma\otimes_\Z\Z_l$, since $p$ is a unit in $\Z_l$.
The lemma follows from the analogous assertions for $\CO$ as in Lemma 2.2 of \cite{class}.
\end{proof}

For a given degree $d$ number field $F$ which embeds into $D(\Q)$, let $S_\inert(F)$ be the set of all $l\in S$ which are inert in $F$.
Define the \e{$S$-inertia degree} by
$$
f_S(F)=\prod_{l\in S}f_l(F)=2^{|S_\inert(F)|},
$$
where $f_l(F)$ is the inertia degree of $l$ in $F$.
For any order $\La$ of $F$ let
$$
f_S(\La)=f_S(F).
$$
Let $\La$ be a $\Z[1/p]$-order in $F$, which is maximal at all $l\in S$.
By Lemma \ref{lem1.2} there exists an embedding $\sigma:F\to D(\Q)$ such that $\La=\La_\sigma$.
Let $u\in D(\Z[1/p])^\times$ and let $^u\sigma$ be the embedding given by $^u\sigma(x)=u\sigma(x)u^{-1}$.
Then $\La_{^u\sigma}=\La_\sigma$, so the group $D(\Z[1/p])^\times$ acts on the set $\Sigma(\La)$ of all $\sigma$ with $\La_\sigma=\La$.

\begin{lemma}\label{lem1.3}
The quotient $\Sigma(\La)/D(\Z[1/p])^\times$ is finite and has cardinality 
$$
\left|\Sigma(\La)/D(\Z[1/p])^\times\right|=f_S(\La)h(\La).
$$
Further one has
$$
\left|\Sigma(\La)/D(\Z[1/p])^1\right|=d_p(\La)f_S(\La)h(\La),
$$
where 
$$
d_p(\La)=\begin{cases}1&p\in N(\La),\\
d &\text{otherwise.}\end{cases}
$$
\end{lemma}

xxx

Compare Lemma 2.3 in \cite{class}.

\begin{proof}
Fix an embedding $F\hookrightarrow D(\Q)$ and consider $F$ as a subfield of $D(\Q)$ such that $\La=F\cap D(\Z[1/p])$. 
For $u\in D(\Q)^\times$ let
$$
\La_u= F\cap u^{-1} D(\Z[1/p])u.
$$
Let $U$ be the set of all $u\in D(\Q)^\times$ such that $\La_u=\La$, i.e.,
$$
F\cap D(\Z[1/p])=F\cap u^{-1}D(\Z[1/p])u.
$$
Then $F^\times$ acts on $U$ by multiplication from the right and $D(\Z[1/p])^\times$ acts by multiplication from the left.
One has
$$
|D(\Z[1/p])^\times\bs U/F^\times|=|D(\Z[1/p])^\times\bs \Sigma(\La)|.
$$
So we have to show that the left hand side equals $f_S(\La)h(\La)$.
For $u\in U$ let
$$
I_u=F\cap D(\Z[1/p])u.
$$
Then $I_u$ is a finitely generated $\La$-module in $F$.
We claim that the map
\begin{align*}
\psi:D(\Z[1/p])^\times\bs U/F^\times &\to I(\La)/F^\times,\\
u&\mapsto I_u
\end{align*}
is surjective and $h(\La)$ to one.
We show this by localization and strong approximation.
For any prime $l\ne p$ let $U_l$ be the set of all $u_l\in D(\Q_l)$ such that $\La_l=F_l\cap D(\Z_l)= F_l\cap u_l^{-1}D(\Z_l)u_l$.
We have to show the following:
\begin{enumerate}[\rm (a)]
\item For $l\notin S$, the localized map $\psi_l:D(\Z_l)^\times\bs U_l/F_l^\times\to I(\CO_l)/F_l^\times$ is injective,
\item for $l\in S$, the map $\psi_l$ is $f_l(F)$ to one,
\item the map $\psi$ is surjective.
\end{enumerate}
For (a) let $l\notin S$, $u_l,v_l\in U_l$ and assume
$$
F_l\cap D(\Z_l)u_l=F_l \cap D(\Z_l)v_l.
$$
Let $z_l=v_lu_l^{-1}$.
Elementary divisor theory implies that there exist
$x,y\in D(\Z_l)^\times=\M_d(\Z_l)^\times$ such that
$$
z_l=x\mathrm{diag}(l^{k_1},l^{k_2})y
$$
holds, where $k_1\le k_2$.
Replacing $u_l$ by $yu_l$ and $v_l$ by $x^{-1} v_l$ we may assume that $z_l$ equals the diagonal matrix.
The assumptions then  imply $k_1=0=k_2$, which gives the first claim.
For (b) let $l\in S$ and recall that $F_l$ is a local field, so $h(\La_l)=1$.
Hence the claim is equivalent to
$$
|D(\Z_l)^\times\bs D(\Q_l)^\times/F_l^\times|=f_l(F).
$$
Taking the $l$-valuation of the reduced norm, one sees that the left hand side equals $2$ if $F_l$ is unramified over $\Q_l$ and $1$ otherwise, i.e., it equals the inertia degree $f_l(F)$ as claimed.

Finally, for the surjectivity of $\psi$ let $I\subset\La$ be an ideal.
We show that there is $u\in D(\Q)^\times$ such that
$$
F\cap u^{-1} D(\Z[1/p])u= F\cap D(\Z[1/p])
$$
and 
$$
I=I_u=F\cap D(\Z[1/p])u.
$$
We do this locally.
First note that, since $I$ is finitely generated, there is a finite set $T$ of primes with $T\cap S=\emptyset$ and $p\notin T$ such that for any $l\notin T\cup S$ the completion $I_l$ equals $\La_l$ which is the maximal order of $F_l$.
For these $l$ set $\tilde u_l=1$.
Next let $l\in S$ and write $v_l$ for the unique place of $F$ over $l$.
Then $\La_l$ is maximal, so is the valuation ring to $v_l$ and $I_l=\pi_l^k\CO_l$ for some $k\ge 0$, where $\pi_l$ is a uniformizer at $v_l$.
In this case set $\tilde u_l=\pi_l^k$.

Next let $l\in T$. Then $D(\Z_l)=\M_d(\Z_l)$. Let $\ol{\La_l}=\La_l/l\La_l$ and $\ol{I_l}=I_l/lI_l$.
Then $\ol{\La_l}$ is a commutative algebra over the field $\F_l$ with $l$ elements, which implies that $\ol{\La_l}\cong\bigoplus_{i=1}^s F_i$, where each $F_i$ is a finite field extension of $\F_l$.
Let $n_i$ be its degree.
Then there is an embedding $\ol{\La_l}\hookrightarrow \M_d(\F_l)$ whose image lies in $\M_{n_1}(\F_l)\times\dots\times\M_{n_s}(\F_l)$. 
By the Skolem-Noether Theorem there is a matrix $\ol S\in \GL_d(\F_l)$ such that $\ol S\,\ol{\La_l}\,\ol S^{-1}\subset \M_{n_1}(\F_l)\times\dots\times\M_{n_s}(\F_l)$.
The $\ol{\La_l}$-ideal $\ol{I_l}$ must be of the form
$$
\ol{I_l}=\bigoplus_{i=1}^s\eps_i F_i,
$$
where $\eps_i\in\{ 0,1\}$.
Let $S$ be a matrix in $\GL_d(\Z_l)$ which reduces to $\ol S$ modulo $l$ and let $\tilde u_l=S^{-1}(l^{\eps_1}\Id_{n_1}\times\dots\times l^{n_s}\Id_{n_s})S$ in $\M_d(\Z_l)$.
By abuse of notation we also write $\tilde u_l$ for its reduction modulo $l$.
Then we have
$$
\ol{I_l}=\ol{\La_l}\cap\M_d(\F_l)\tilde u_l.
$$
Let 
$$
I_{\tilde u_l}= F\cap D(\Z_l)\tilde u_l.
$$
Then it follows that
$$
\ol{I_l}\cong \ol{I_{\tilde u_l}}=I_{\tilde u_l}/lI_{\tilde u_l}
$$
and by Theorem 18.6 of \cite{Pierce} we get that $I_l\cong I_{\tilde u_l}$, which implies that there is some $\la\in F_l$ with $I_l=I_{\tilde u_l}\la$.
Replacing $\tilde u_l$ by $\tilde u_l\la$ and setting $\tilde u=(\tilde u_l)_l\in D(\A_\fin)$ we get
$$
I=F\cap D(\Z[1/p])\tilde u.
$$
By strong approximation there is an element $u\in D(\Q)^\times$ such that $D(\hat\Z)u=D(\hat\Z)\tilde u$ and therefore $I=I_u$.
\end{proof}

\section{The zeta function}
We now restrict to the case $d=2$, so $D$ is a quaternion division algebra over $\Q$, and we assume that the primes 2 and 3 decompose in $D$.
We further assume that $D$ does not split at $\infty$.
For any ring $R$ we write $\det:D(R)\to R$ for the reduced norm.
Note that this convention is compatible with the determinant, as for every field $F$, over which $D$ splits, the reduced norm equals the determinant.
We want to construct a group scheme $\CG$ over $\Z$ such that $\CG(F)=D(F)^\times/F^\times$ holds for every field.
For this recall that we have fixed a maximal order $D(\Z)$ in $D$.
Note that $D(\Z)$ is a free $\Z$-module of rank $4$.
Let $v_1,\dots v_4$ be a basis and note that the reduced norm
$\det(X_1v_1+\dots+X_4v_4)$ is a homogeneous polynomial of degree $2$ in the variables $X_1,\dots,X_4$.
The group scheme $D^\times$ is given by the coordinate ring
$$
\CO_{D^\times}=\Z[X_1,X_2,X_3,X_4,Y]/\(\det(X_1v_1+\dots+X_4v_4)Y-1\).
$$
Now $\GL_1$ acts on $\CO_{D^\times}$ by
$$
\al f(X_1,X_2,X_3,X_4,Y)=f(\al X_1,\al X_2,\al X_3,\al X_4,\al^{-2}Y)
$$
and the coordinate ring we need is the ring of invariants
$$
\CO_\CG=\(\CO_{D^\times}\)^{\GL_1},
$$
which is the subring generated by the elements $X_iX_iY$ with $1\le i\le j\le 4$.

\begin{lemma}
The ring $\CO_\CG$ is the coordinate ring of an affine group scheme $\CG$ over $\Z$ such that for every factorial ring $R$ one has
$$
\CG(R)= D(R)^\times/R^\times.
$$
\end{lemma}

\begin{proof}
The first claim is clear. We prove the second first in case of a field.
Consider the exact sequence of group schemes
$$
1\to \GL_1\to D^\times\to\CG\to 1.
$$
For any field $K$ this gives an exact sequence of groups
$$
1\to \GL_1(K)\to D(K)^\times\to\CG(K)\to H^1(K,\GL_1),
$$
where the last item is the Galois-cohomology, which vanishes by Hilbert's Theorem 90.
This implies the claim for fields.
Now let $R$ be a factorial ring, so $R$ is integral and has unique factorization.
Write $K$ for its quotient field and let $\chi\in \CG(R)$, so $\chi$ is a ring homomorphism from $\CO_\CG$ to $R$.
By the first part of the proof, $\chi$ extends to a ring homomorphism $\tilde\chi:\CO_{D^\times}\to K$.
We show that this lift can be modified so as to have values in $R$.
For $1\le i\le 4$ we have
$$
\tilde\chi(x)^2\tilde\chi(y)=\chi(x^2y)\in R.
$$
If $p$ is an irreducible in $R$ which divides the denominator of, say,  $\chi(x_1)$, we replace any $\tilde\chi(x_i)$ by $p\tilde\chi(x_i)$ and $\tilde\chi(y)$ by $\frac1{p^2}\tilde\chi(y)$ without changing $\chi$, so we can assume $\tilde\chi(x_1)$ to lie in $R$.
We repeat this with $x_2$ and so on.
Now if $\tilde\chi(y)$ does not lie in $R$ there must be an irreducible $p$ dividing its denominator. But as the product is in $R$, $p$ also divides $\tilde\chi(x_i)^2$ hence $\tilde\chi(x_i)$.
So we can replace $\tilde\chi(x_i)$ by $\frac1p\tilde\chi(x_i)$ and $\tilde\chi(y)$ by $p^2\tilde\chi(x_i)$ and by repeating this procedure we arrive at $\tilde\chi(x_i)$ and $\tilde\chi(y)$ both lying in $R$.
\end{proof}

We set $\Ga=\CG(\Z[1/p])$.
By Theorem 3.2.4 in \cite{Margulis}, the group $\Ga$ is a uniform lattice in $G=\CG(\Q_p)\cong \PGL_2(\Q_p)$, i.e., $\Ga$ is a discrete subgroup of $G$ such that $\Ga\bs G$ is compact.

\begin{lemma}
The group $\Ga$ is torsion-free.
\end{lemma}

\begin{proof}
$\Ga$ is a subgroup of $D(\Q)^\times/\Z[1/p]^\times$ which has no torsion elements by Lemma \ref{lem1.1}.
\end{proof}

Keep the prime $p\notin S$ fixed and let $Y$ be the Bruhat-Tits building of $G=\CG(\Q_p)$.
Then $Y$ is a regular tree of valency $p+1$, \cite{Serre}.
The group $\Ga$ acts freely and discontinuously  on the contractible space $Y$, so $\Ga$ is the fundamental group of the finite graph $X=\Ga\bs Y$. 

\begin{lemma}
No element $\ga\ne 1$ of $\Ga$ is conjugate to its inverse.
\end{lemma}

\begin{proof}
As $\Ga$ is the fundamental group of the graph $X$, its conjugacy classes are in bijection with the homotopy class of loops $S^1\to X$.
Not non-trivial loop in a graph is homotopic to its inverse.
\end{proof}

Every element $\ga\in\Ga\sm\{ 1\}$ closes a geodesic on $X$, the length of which equals
$$
l(\ga)=\min_{y\in Y}d(\ga y,y),
$$
where the minimum is taken over all vertices $y$ in the tree $Y$ and $d(\ga y,y)$ is the graph distance.
Consider $\ga$ as an element of $G=\PGL_2(\Q_q)$.
A \e{$\Z_q$-lattice} in $\Q_q^2$ is a finitely generated $\Z_p$-submodule which generates $\Q_q^2$ as a $\Q_q$-vector space.
The multiplicative group $\Q_p^\times$ acts on the set of all $\Z_p$-lattices by multiplication.
The vertices of the tree $Y$ can be identified with the set of $\Q_p^\times$-orbits of lattices.
Two lattice classes $[L]$ and $[L']$ are connected by an edge  if and only if there are representatives with $pL\subset L'\subset L$.
It follows that the graph distance between two lattice classes $[L]$ and $[L']$ is the smallest $k\in\N_0$ such that there are representatives with $p^kL\subset L'\subset L$.
So let $\ga\in\Ga$ and let $L$ be a lattice such that $l(\ga)=d(\ga [L],[L])$.
So we can assume that $\ga$ is represented by an element $\tilde\ga\in D(\Z[1/p])^\times\subset \GL_2(\Q_p)$, such that $p^{l(\ga)}L\subset \tilde\ga L\subset L$.
By elementary divisor theory there exists a basis $e_1,e_2$ of $L$ such that $L=\Z_p e_1\oplus\Z_pe_2$ and
$\tilde\ga L=\Z_p e_1\oplus p^{l(\ga)}\Z_pe_2$.
In other words, there are $A,B\in\GL_2(\Z_p)$ such that
$$
\tilde\ga=A\mat 1\ \ {p^{l(\ga)}}B.
$$
Or, again in other words, with $F_\ga$ being the centralizer of $\tilde\ga$ in $D(\Q)$ and $\CO_\ga$ denoting the order $F_\ga\cap D(\Z)$ we get that $N_{F_\ga/\Q}(\tilde\ga)=p^{l(\ga)}m$ with $m$ coprime to $p$.
On the other hand, as $\tilde\ga\in\CO_\ga^\times$, its norm needs to be in $\Z[1/p]^\times$ and so $m=\pm 1$ and as $F$ is imaginary quadratic, finally $m=1$.

An element $\ga$ of $\Ga$ is called \e{primitive}, if $\ga=\sigma^n$ for $\ga\in\Ga$ and $n\in\N$ implies $n=1$.
Note that every element $\ga\ne 1$ of $\Ga$ is a positive power of a unique primitive element.
From the above it follows that if $\ga$ is primitive, then 
$$
l(\ga)=\min
\{ k\in\N: p^k\in N_{F/\Q}(\CO_\ga\sm\Q)\},
$$
Note that if $F$ is a quadratic subfield of $D(\Q)$, then $F\otimes_\Q\R$ injects into the division algebra $D(\R)$, which implies that $F$ is an imaginary quadratic field.

Let $L(S)$ be the set of isomorphy classes of $\Z[1/q]$-orders $\La$ in imaginary quadratic fields $F$, such that every $p\in S$ is non-decomposed in $F$ and $\La$ is maximal at every $p\in S$.
We just found a natural map
$$
\Psi:\Ga/\text{conjugation}\to L(S)
$$
given by $\psi(\ga)=\La_\ga$.

For a given $\Z[1/q]$-order $\La$ in $F$ let $\CO_\La=\La\cap\CO_F$ and let
$$
l_q(\La)=\min\{l\in\N: q^l\in N_{F/\Q}(\CO_\La\sm\Z)\},
$$
where the minimum of the empty set is $+\infty$ and $N_{F/\Q}$ is the norm map.

\begin{lemma}\label{lem2.4}
The image of the map $\Psi$ consists of all classes $[\La]\in L(S)$ with $l_q(\La)<\infty$.
Each such class $[\CO]$ in $L(S)$ has 
$$
2f_S(\La)h(\La)
$$
many preimages.
\end{lemma}

\begin{proof}
Let $[\La]$ be a class in $L(S)$ with $l_q(\La)<\infty$. 
We can assume that $\La=F\cap D(\Z[1/q])=(F\cap D(\Z))[1/q]$.
By Lemma \ref{lem1.3}, any given $\Z[1/q]$-order $\La$ comes about in this way with $f_S(\La)h(\La)$ many different embeddings up to conjugacy, each embedding gives two generators $\ga$ and $\ga^{-1}$,
 which are not conjugate, therefore we get the claimed number of preimages.
\end{proof}

\begin{theorem}
Let $\Delta$ be the greatest common divisor of the numbers $l_p(\La)$ as $\CO$ ranges over $L(S)$.
Then one has
$$
\sum_{\substack{\La\in L(S)\\ l_p(\La)=\Delta m}}f_S(\La)h(\La)\quad\sim\quad \frac{p^{\Delta m}}{2m}
$$
as $m\to\infty$.
\end{theorem}

\begin{proof}
With Lemma \ref{lem2.4} the theorem  follows from the prime geodesic theorem as in Section 2.7 of \cite{Terras}.
\end{proof}

{\small Mathematisches Institut\\
Auf der Morgenstelle 10\\
72076 T\"ubingen\\
Germany\\
\tt deitmar@uni-tuebingen.de}

\end{document}